\documentclass{amsart} 
\usepackage{color}
\usepackage{amssymb,graphicx}

\def\ssm{\smallsetminus}

\def\Z{\mathbb{Z}}

\def\<{\langle}
\def\>{\rangle}
\def\-{\overline}

\newtheorem{thmA}{Theorem}

\newtheorem{theorem}{Theorem}[section]
\newtheorem{lemma}[theorem]{Lemma}

\newtheorem{proposition}[theorem]{Proposition}

\newtheorem{corollary}[theorem]{Corollary}

\theoremstyle{definition} 
 
\newtheorem{remarks}[theorem]{Remarks}
\newtheorem{remark}[theorem]{Remark}

\begin{document}
 
\catcode`\@=11
\def\serieslogo@{\relax}
\def\@setcopyright{\relax}
\catcode`\@=12

\textheight 18cm 

\title[Complete Embeddings of Groups]{Complete Embeddings of Groups}

\author[Bridson]{Martin R.~Bridson}
\address{Martin R.~Bridson\\
Mathematical Institute \\
Andrew Wiles Building \\
ROQ, Woodstock Road\\
Oxford OX2 6GG\\
Europe }
\email{bridson@maths.ox.ac.uk}

 \author[Short]{ Hamish Short }
\address{ Hamish Short \\
L.A.T.P., U.M.R. 6632 \\
C.M.I.,
39 Rue Joliot--Curie\\
Universit\'e de Provence, F--13453\\
Marseille cedex 13, France }
\email{ hamish@cmi.univ-mrs.fr } 



\begin{abstract} Every countable group $G$ can be embedded in a finitely generated group $G^*$ that is hopfian and
{\em complete}, i.e.~$G^*$ has trivial centre and every
epimorphism $G^*\to G^*$ is an inner automorphism. Every finite subgroup of $G^*$
is conjugate to a finite subgroup of $G$. If $G$ has a  finite presentation (respectively,  a finite classifying space),
then so does $G^*$. Our construction of $G^*$ relies on the existence of closed hyperbolic
3-manifolds that are asymmetric and non-Haken.
\end{abstract}

\maketitle

\begin{center}
{\em For our friend and coauthor Chuck Miller in his ninth decade, \newline
 with deep respect and affection}
\end{center}

\section*{Introduction}

In 1971, Charles F. Miller III and Paul E. Schupp \cite{MS} used small cancellation theory to prove
 that every countable group $G$ can be embedded in  a finitely generated group $G^*$ that is {\em hopfian} and {\em complete} (asymmetric). 
They construct $G^*$ as a quotient of a free product
$G\ast U(p,q)$, where $U(p,q)$ is a free product of finite cyclic groups. In particular, each of the enveloping groups that
they construct has torsion. 
The purpose of this note is to present an alternative construction of $G^*$ that does not 
introduce torsion.
It also preserves finiteness properties of $G$.

\begin{thmA}\label{thm} Every countable group $G$ can be embedded in a finitely generated group $G^*$ such that 
\begin{enumerate}
\item $G^*$ is hopfian and complete;
\item every finite subgroup of $G^*$
is conjugate to a finite subgroup of $G$;
\item if $G$ has a finite presentation (respectively, a finite classifying space of dimension $d\ge 3$),
then so does $G^*$.
\end{enumerate}
\end{thmA}  

Our construction of $G^*$ 
is more explicit than that of Miller and Schupp. It is non-trivial
in that it relies on 
the existence of asymmetric hyperbolic groups with additional properties,
but the basic idea behind it is straightforward: after some gentle preparation,
we are able to assume that $G$ is generated by two finitely generated free subgroups $F_1,F_2<G$;
we then {\em rigidify}  $G$ by attaching certain complete (asymmetric)
groups $A_1$ and $A_2$ to it
along $F_1$ and $F_2$; the groups $A_i$ are constructed so that any epimorphism
$\psi$
of the resulting amalgam $G^*$ must preserve the decomposition 
$G^* = A_1\ast_{F_1} G \ast_{F_2} A_2$, sending each $A_i$ to itself; by ensuring
that the $A_i$ are hopfian, as well as complete, we force the restrictions $\psi|_{A_i}$ to be the identity; and since $G^*=\<A_1, A_2\>$, we conclude that $\psi$ is the identity.

Groups $A_i$ with the properties that we need can be found among the fundamental groups of closed hyperbolic 3-manifolds obtained by Dehn surgery on the knots shown in figure 1, as we shall explain.

\section{Preliminaries}

We remind the reader of some terminology.
A group $G$ is termed {\em hopfian} if every epimorphism
$G\to G$ is an isomorphism, and co-hopfian is every monomorphism $G\to G$
is an isomorphism. $G$ is said to be {\em complete} (or {\em asymmetric})
if it has trivial centre and every automorphism is inner. A subgroup $H<G$
is {\em malnormal} if $H^g\cap H \neq 1$ implies $g\in H$. A 
group $G$ is said to {\em split over a free group} if $G$ 
can be decomposed as an amalgamated free product $G=A\ast_F B$ or HNN extension
$G=B\ast_F$ with $F$ free. A group has Serre's property {\rm{FA}} if it fixes
a point whenever it acts on a simplicial tree.

We shall assume that the reader is familiar with 
the rudiments of Bass-Serre theory \cite{serre} and the homology of groups.

\subsection*{The Mild Preparation of $G$}

A classical construction of B.H.~Neumann embeds a countable group $G$ in a finitely
generated group $\widetilde{G}$ by means of
HNN extensions and amalgamated free products 
(see \cite[page 188]{LS}).  
The finite subgroups of $\widetilde{G}$ are conjugate to subgroups of $G$;  in particular
$\widetilde{G}$  is torsion-free if $G$ is torsion-free. Thus, in our attempts to construct $G^*$,
there is no loss of generality in assuming that $G$ is finitely generated. 
 
Replacing
$G$ by $G\ast\Z$ if necessary, we may also assume that our group has a generating set 
$\{a_0,a_1,\dots,a_n\}$ where the $a_i$ each have infinite order: given
$G=\<b_1,\dots,b_n\>$ and
$x$ generating $\Z$, define $a_0=x$ and $a_i=xb_i$. We assume that this modification
has been made.
The normal form theorem for free products then yields:

\begin{lemma}\label{l:2free} In $G\ast\<s,t\>$, the subsets $\{s,a_0t, \dots, a_n(s^{-n}ts^n)\}$ and $\{s,t,a_0t\}$
generate free subgroups of ranks $(n+2)$ and $3$, respectively.
\end{lemma}

\begin{corollary}\label{r:f12}
Replacing $G$ by $G\ast\<s,t\>$ if necessary, we may assume that $G = \<F_1,F_2\>$, where $F_1$ and $F_2$ are finitely
generated free groups with non-cyclic intersection and the centralizer of $F_1\cap F_2$ in $G$ is trivial.
\end{corollary}
 
\subsection*{Asymmetric Hyperbolic Manifolds}

\begin{lemma}\label{l:asymm} Given integers $r>0$ and $d\ge 3$, one can construct a complete, torsion-free, co-Hopfian
(Gromov) hyperbolic
group $A(d)$ and a malnormal free subgroup $L<A(d)$ of rank $r$ such that
\begin{enumerate}
\item $A(d)$ has a finite classifying space of dimension $d$,
\item $A(d)$ does not split over any free group, and 
\item $A(d)/\<\!\<L\>\!\>$ is infinite.
\end{enumerate} 
\end{lemma}

\begin{proof} By Mostow
rigidity, if $M$ is a closed orientable hyperbolic manifold of dimension $d\ge 3$,
then $\pi_1M$ is complete (asymmetric)
if and only if $M$ is asymmetric, i.e.~$M$ has
no non-trivial isometries. There exist such manifolds in every dimension $d\ge 3$: Kojima \cite{koji}
constructed examples in dimension
$3$ (along with manifolds that have any prescribed finite group of symmetries)  and, inspired by arguments of Long and Reid \cite{LR}, 
Belolipetsky and Lubotzky \cite{BL} constructed examples in each dimension $d\ge 3$. 
Let $M$ be such a $d$-manifold, and let $A(d)=\pi_1M$. Note 
that $M$ is a classifying space for $A(d)$.

If $A(d)$ splits non-trivially as an amalgamated free product, 
say $A(d)= D\ast_C B$, then there is a Mayer–Vietoris exact sequence for integral homology groups
$$
\cdots H_dD\oplus H_dB \to H_dM \to H_{d-1}C \cdots
$$ 
By Poincar\'e duality, $H_dM\cong \mathbb{Z}$
and (as $D, B < \pi_1M$ are of infinite index) $H_dD = H_dB = 0$.
Thus $H_{d-1}C$ is infinite. 
In particular, since $d\ge 3$, the group $C$ cannot be free. 
A similar argument shows that $\pi_1M$
does not split as an HNN extension over a free group either.

Any subgroup of infinite
index in $A(d)$ has lesser cohomological dimension than $A(d)$, and a subgroup of finite index cannot be isomorphic to $A(d)$ by Mostow rigidity. 
Thus $A(d)$ is co-Hopfian. 

Every non-elementary 
hyperbolic group contains proper, normal subgroups of infinite index, 
and Kapovich \cite{K} shows that inside such a
subgroup one can find a malnormal free subgroup of rank $2$, and inside that one can find a malnormal subgroup of any finite rank. 
\end{proof}

\section{The main argument}
 
Given a countable group $G$, we modify it  
to arrive in the situation $G = \<F_1,F_2\>$ described in Corollary \ref{r:f12}.
Let $r_i$ be the rank of the free group $F_i$ and construct groups $A_i=A(d_i)$ with malnormal
free subgroups $L_i<A_i$ of rank $r_i$ as in Lemma \ref{l:asymm}. We may assume that
$A_1$ is not isomorphic to $A_2$ (Remark \ref{r22}).
Let 
\begin{equation}\label{G*}
G^* = A_1\ast_{F_1} G\ast_{F_2} A_2,
\end{equation}
where the amalgamation identifies $F_i< G$ with $L_i<A_i$.
Note that since $G=\<F_1,F_2\>$, we have $G^*=\<A_1,A_2\>$.

\begin{theorem}\label{t:rigid}
$G^*$ is complete.
\end{theorem}

\begin{proof} The centre of $G^*$ is trivial because the centre of any amalgamated
free product lies in the intersection of the edge groups, and the edge groups in the 
defining decomposition of $G^*$ are centreless. 

Let $\phi : G^*\to G^*$ be an automorphism; we must argue that $\phi$ is inner.
As $A_i$ does not split over a free group, the action of $\phi(A_i)$ on the Bass-Serre tree of the given splitting of $G^*$
must fix a vertex. Thus each of $\phi(A_1)$ and $\phi(A_2)$ is contained in a conjugate of $A_1, A_2$ or $G$.

 We have chosen
the $A_i$ so that $Q_i:=A_i/\<\!\<F_i\>\!\>$ is infinite. Note that $Q_1\cong G^*/\<\!\<G, A_2\>\!\>$ 
and $Q_2\cong G^*/\<\!\<G, A_1\>\!\>$. Thus a pair of conjugates of $A_1, A_2$ or $G$ can  only generate $G^*$
if one of the pair is a conjugate of $A_1$ and the other is a conjugate of $A_2$. 
If $\phi$ maps $A_1$ into a conjugate of $A_2$, then 
$\phi^2$ would map $A_1$ to a conjugate of itself, and since 
$A_1$ is co-Hopfian, the image would be the whole of this conjugate. This forces
$\phi(A_1)$ to be equal to the conjugate of $A_2$ containing it, which is impossible since we have chosen
$A_1$ and $A_2$ to be not isomorphic. We conclude that $\phi$
maps $A_i$ isomorphically onto a conjugate of itself for $i=1,2$.

Now, since  all automorphisms of $A_1$ are assumed to be inner, 
we can compose $\phi$ with
an inner automorphism of $G^*$ to assume that $\phi|_{A_1}={\rm{id}}_{A_1}$,
while $\phi(A_2)=A_2^\gamma$ for some $\gamma\in G^*$. 

Consider the Bass-Serre tree for the splitting 
$G^* = A_1\ast_{F_1} G\ast_{F_2} A_2.$ We refer to the vertices as being of
type $A_1, A_2$ or $G$, according to whether they are in the 
$G^*$ orbit of the vertices (identity cosets) $A_1, A_2$ or $G$, respectively.
Since $F_i<A_i$ is malnormal for $i=1,2$, no arc of length 
greater than $2$ in this  tree has non-trivial
stabilizer, and any arc of length $2$
with non-trivial stabilizer must be centred at a vertex of type $G$. In particular, the
subtree fixed by $F_1\cap F_2$, which contains the vertices $A_1, G$ and $A_2$, has
diameter $2$ and centroid $G$. The centraliser of $F_1\cap F_2$ in $G^*$  leaves this subtree and
its centroid invariant, and hence is contained in $G$. But, by construction, the centraliser of 
$F_1\cap F_2$ in $G$ is trivial, and hence so is its centraliser in $G^*$.

As $\phi(A_2)=A_2^\gamma$, 
we have an isomorphism ${\rm{ad}}(\gamma)^{-1}\circ \phi|_{A_2}:A_2\to A_2$. As $A_2$ is complete,
this isomorphism is conjugation by some $a\in A_2$, so  $\phi|_{A_2}$ is conjugation by $\gamma a\in G^*$. 
But $\phi$ restricts to the identity 
on ${F_1\cap F_2}<A_2$, and we know that
the centralizer of $F_1\cap F_2$ in $G^*$ is trivial, so  $\gamma=a^{-1}$ and $\phi|_{A_2}={\rm{id}}_{A_2}$.
  
As $G^*$ is generated by $A_1\cup A_2$, we conclude that $\phi$ 
(previously adjusted by a conjugacy to ensure that $\phi_{A_1}={\rm{id}}_{A_1}$) 
is the identity, and the theorem is proved.
\end{proof}

\begin{remark}\label{r22}
 In the theorem above, we required $A_1\not\cong A_2$. An easy way to arrange this is to take $A_i=A(d_i)$ from Lemma \ref{l:asymm} with $d_1\neq d_2$.
But one is also free to take both $A_i$ to have the same dimension $d\ge 3$, appealing to
\cite{BL} or Theorem \ref{p:3hopf} below. 
With this second choice, if $G$ has 
geometric (or cohomological) dimension $D$, then $G^*$ will have
geometric (or cohomological) dimension $\max\{D, d\}$.
\end{remark}

\section{Asymmetric hyperbolic 3-manifolds}\label{s:3folds}

A closed orientable 3-manifold $M$ is {\em{Haken}} if it is irreducible   
and contains a closed incompressible
surface, i.e. a closed surface 
of positive genus such that the inclusion map $S\hookrightarrow M$ induces a
monomorphism of groups $\pi_1S\hookrightarrow \pi_1M$. If $M$ contains
such a surface,
then $\pi_1M$ acts without a fixed point on the tree $T$ that is obtained from
the universal covering $p:\tilde M\to M$ as follows: the
vertex set of $T$ is 
the set of connected components of $\tilde M \ssm p^{-1}(S)$; two vertices are connected by an edge if the components that they represent abut along a component of $p^{-1}(S)$; and the action of $\pi_1M$ on $T$ is induced by the action
of $\pi_1M$ on $\tilde M$ by deck transformations. 

In the opposite direction, John Stallings proved that
if $M$ is irreducible and $\pi_1M$ acts without a fixed point on a simplical tree, then $M$ contains an
incompressible surface; see \cite{stall}. In particular, a closed orientable 3-manifold $M$ that is  aspherical
will be non-Haken if and only if $\pi_1M$ has Serre's property {\rm{FA}}. 

Our purpose in this section is to explain how well-known facts about hyperbolic
3-manifolds imply the following result, which will be familiar to experts.

\begin{theorem} \label{p:3hopf}
There exist infinitely many distinct, closed, asymmetric,  hyperbolic 3-manifolds
$M$ such that $\pi_1M$ has property {\rm{FA}}.
\end{theorem}

\begin{proof} In the light of the preceding discussion, what we must show is that 
there are infinitely many distinct hyperbolic 3-manifolds that are asymmetric and
non-Haken. The manifolds that we shall describe are obtained by Dehn surgery on 
the knots shown in figure 1. These are the only four knots $K$
with at most $10$ crossings
that have the three properties that we are interested in: first, each $K$ is  prime and
alternating (and non-torus), hence hyperbolic \cite{menasco}; second, the knot complement
$\mathbb{S}^3\smallsetminus K$ has no non-trivial symmetries, so by Mostow rigidity its
fundamental group 
is complete (asymmetric); and third, $\mathbb{S}^3\smallsetminus K$ is {\em{small}},
i.e. contains no
closed incompressible surface other than the tori parallel to 
the boundary of a regular neighbourhood of $K$. The first of these properties is immediately visible
in the diagrams, the second is established in the census of Henry and Weeks \cite{HW} 
who calculated
the symmetry groups of all knots up to 10 crossings (alternatively
Kodama and Sakuma \cite{KS}), and the third is recorded
in the census \cite{BCT} of Burton, Coward and Tillmann,
who calculated all knots with at most 12 crossings that are small.

\begin{figure}[h]
\begin{center} 
	\includegraphics[scale=0.35]{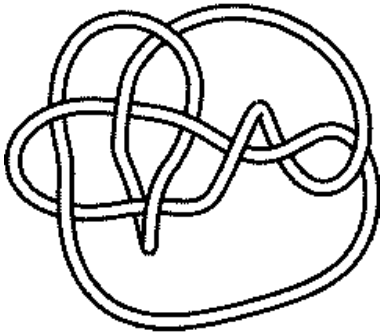} 
	\includegraphics[scale=0.35]{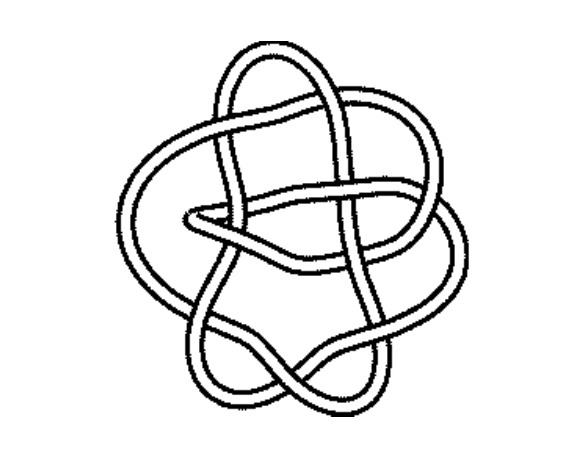}
	\includegraphics[scale=0.35]{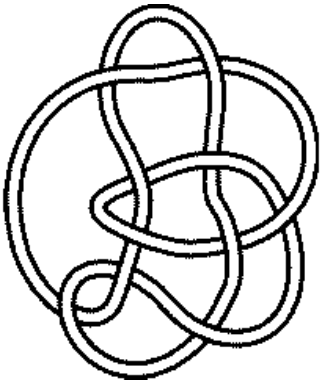}
	\includegraphics[scale=0.35]{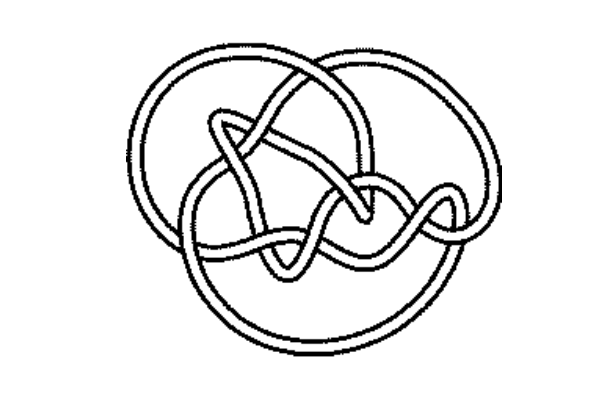}
 \caption{The knots $10_{102}, 10_{106}, 10_{107}$ and $10_{110}$}
 \label{fig1}
\end{center}
\end{figure}

Thurston's celebrated Dehn surgery theorem \cite{wpt} states that all but finitely
many of the closed manifolds obtained by Dehn surgery on
these knots will support a hyperbolic
metric.  And Hatcher's theorem on boundary slopes
\cite{hatcher} implies that all but finitely many of these closed manifolds will
be non-Haken. Thus we will be done if we can argue that all but finitely many 
of these closed hyperbolic manifolds are asymmetric. And this is what 
Kojima's arguments show \cite{koji}, as we shall now explain.

In Dehn filling, 
one starts from the manifold $M_K$ with torus boundary obtained by removing an open
tubular neighbourhood of a hyperbolic knot $K$ in  $\mathbb{S}^3$. A framing of the knot
gives an identification 
$\Z^2=\pi_1\partial M_k$. Given a reduced fraction $p/q$, one attaches
a thickened 2-disc $D$ to an annular neighbourhood of a simple closed curve on 
$\partial M_K$ representing the homotopy class  $(p,q)\in \Z^2$. The boundary of the resulting manifold $M_K\cup D$ is a 2-sphere, which one caps off with a 3-ball to
obtained the  manifold $M_K(p,q)$. (Note that the union of $D$ and this
3-ball is a solid torus.)  For all but finitely many choices
of $p/q$, Thurston \cite{wpt} provides a hyperbolic metric on $M_K(p,q)$, and by Mostow rigidity
this metric is unique.  
As $|p|+|q|\to \infty$  this metric  
converges to the complete hyperbolic metric on $M_K$, and when $|p|+|q|$ is sufficiently large, the
unique shortest geodesic in  $M_K(p,q)$ is the  core of the solid torus added during Dehn filling \cite{wpt}.
Thus, by restriction, we obtain a homomorphism ${\rm{Isom}}(M_K(p,q))\to {\rm{Homeo}}(M_K')$, 
where $M_K'$ is the interior of $M_K$.
By Mostow rigidity, every homeomorphism of $M_K'$ is homotopic to a unique isometry of the
complete hyperbolic metric on $M_K'$, so passing to homotopy classes we have a homomorphism
$\rho:{\rm{Isom}}(M_K(p,q))\to {\rm{Isom}}(M_K')$ (where the isometries of $M_K'$ are with respect to
its complete hyperbolic metric, not the restriction of the metric on  $M_K(p,q)$). 
Mostow rigidity also tells us that, for any complete, finite-volume hyperbolic $3$-manifold $N$,
 the natural map ${\rm{Isom}}(N)\to {\rm{Out}}(\pi_1N)$ is an isomorphism. Thus an isometry $\phi$ of
 $M_K(p,q)$ will lie in the kernel of $\rho$ only if its restriction to $M_K'$ induces an inner automorphism of $\pi_1 M_K$.
 But if this is the case, then the map that $\phi$ induces on $\pi_1M_K(p,q)$ (a quotient of  $\pi_1 M_K'$)
 will also be inner, and therefore $\phi$ is trivial. Thus $\rho:{\rm{Isom}}(M_K(p,q))\to {\rm{Isom}}(M_K')$  
 is injective. In particular, for $K$  with $\pi_1M_K$ complete,  the triviality of ${\rm{Isom}}(M_K')\cong{\rm{Out}}(\pi_1M')$
 implies that 
$ {\rm{Out}}(\pi_1M_K(p,q))\cong 
{ \rm{Isom}}(M_K(p,q))$ is trivial when $|p|+|q|$ is sufficiently large.  
\end{proof}

We shall also need the following lemma. 
Note that a finitely generated group with property FA has finite abelianisation.

\begin{lemma}\label{l:proper}
Let $M$ be a closed hyperbolic 3-manifold. If $H_1M$ is finite, then every
non-trivial homomorphism $f: \pi_1M\to \pi_1M$ is an isomorphism.
\end{lemma}

\begin{proof} 
Scott's compact core theorem \cite{scott} implies that every non-trivial subgroup
of infinite index in $\pi_1M$ has infinite abelianisation, and therefore cannot be
the image of $f$. In more detail, if 
the finitely generated group $G=f(\pi_1M)$ has infinite index, then the
corresponding covering space $M'=\tilde M/G$ of  $M$ has a compact submanifold $C$
such that $C\hookrightarrow M'$ is a homotopy equivalence. A standard argument
using Poincar\'{e}-Lefschetz duality (``half lives, half dies") shows that the rank of $H_1(\partial C)$ is twice the
rank of the image of $H_1(\partial C)$  in $H_1M'$. Since $C$ has at least one component
of positive genus, both $H_1\partial C$ and $H_1M'$ are infinite.

The lemma now follows from the fact that every epimorphism from $\pi_1M$
to a subgroup of finite index in itself is an isomorphism -- see \cite{BHM} or
\cite{sela}. In more detail, an argument of Hirshon \cite{Hirsh} shows that if a finitely
generated, torsion-free group $\Gamma$ is residually finite, then every homomorphism from $\Gamma$
to a subgroup of finite index in itself is injective, and Mostow rigidity implies
that $\pi_1M$ (which is residually finite because it is linear, and torsion-free because $M$ is
aspherical) cannot be isomorphic
to a subgroup of finite index in itself (because such a subgroup has greater covolume).
\end{proof} 

\section{The proof of Theorem \ref{thm}}

Towards proving items (2) and (3) of the theorem,
first note that the process by which a given finitely
generated group $G$ was transformed
into the conditioned state described
in Corollary \ref{r:f12} involved only free products with free groups,
so it preserves the finiteness properties in (3), and any finite subgroup of
the conditioned group is conjugate to a subgroup of the original $G$. This
last property is also true in the case of a countable group that is first embedded
in a finitely generated group using Neumann's embedding (as we noted in section 2).

In the main construction,
we defined $G^* = A_1\ast_{F_1} G\ast_{F_2} A_2$, where the $A_i$ are torsion-free 
and the $F_i$ are finitely generated free groups.  
Every finite subgroup of an amalgamated free product
is conjugate into one of the vertex groups, and
amalgamating groups that are finitely presented (respectively, have a finite
classifying space of dimension at most $d$) along finitely generated free groups preserves these properties.
 Thus (2) is proved, and (3) will be proved if we can argue that both $A_1$ and $A_2$ can be taken
 to be 3-dimensional hyperbolic groups.

In the light of Theorems \ref{t:rigid} and \ref{p:3hopf}, the following proposition
completes the proof of Theorem \ref{thm}.

\begin{proposition} If the complete groups $A_1\not\cong A_2$ used in the construction of $G^*$ 
are the fundamental groups of hyperbolic 3-manifolds that have property ${\rm{FA}}$,  then $G^*$ is hopfian.
\end{proposition} 

\begin{proof} We have $G^* = A_1\ast_{F_1} G \ast_{F_2} A_2$.  
Let $\phi:G^*\to G^*$ be an epimorphism. The action of $\phi(A_1)$ on the Bass-Serre
tree for $G^*$ has a fixed point, as $A_1$ has property {\rm{FA}}, so 
$\phi(A_1)$ lies in a conjugate of one of the vertex groups, $A_1, G, A_2$.
The same argument applies to $A_2$. 

As in the proof of Theorem \ref{t:rigid}, we use the fact that 
a pair of conjugates of $A_1, A_2$ or $G$ can  only generate $G^*$
if one of the pair is a conjugate of $A_1$ and the other is a conjugate of $A_2$. 
From this, and the surjectivity of $\phi$, we deduce that $\phi^2(A_1)$ is a non-trivial
subgroup of a conjugate of $A_1$ and $\phi^2(A_2)$ is a non-trivial
subgroup of a conjugate of $A_2$. Lemma \ref{l:proper} then forces $\phi^2|_{A_1}$
and $\phi^2|_{A_2}$ to be isomorphisms onto conjugates of $A_1$ and $A_2$, 
respectively. We can then proceed exactly as in the proof of Theorem 
\ref{t:rigid} to conclude that $\phi$ is an inner
automorphism of $G^*$.
\end{proof}

\begin{remarks}
(1). Let $G$ be a group that has a finite classifying space and contains
an element $\gamma\neq 1$ conjugate to its inverse.
Our construction embeds $G$ in a group $G^*$ that retains these properties and is
complete. This recovers the main theorem of \cite{BS}. To obtain a specific example,
one can take $G=\<x,y\mid xyx^{-1}y\>$.

\smallskip

(2) A second (less concise and explicit) proof of Theorem \ref{thm} can be obtained 
by means of a careful application of relative small cancellation theory, following
Miller and Schupp's method in \cite{MS} but avoiding the use of finite groups.
In outline,
one constructs a two generator perfect hyperbolic group $B$, takes a free product
$(G\times \Z^2)* (B\times\Z)$  and then forms the
quotient by a set of relators satisfying a strong small cancellation condition \cite{osin}. A key point
in \cite{MS} is that automorphisms  preserve the different finite orders of elements, while here they preserve the different ranks of centralizers of elements.  
\end{remarks}


\begin{thebibliography}{99} 
 


\bibitem{BL}  Mikhail Belilopetsky and Alexander Lubotzky,
{\it Finite groups and hyperbolic manifolds\/}, Invent. Math. {\bf{162}} (2005), 459--472 (2005)


\bibitem{BS} M.R. Bridson and H. Short,
{\em Inversion is possible in groups with no periodic automorphisms},
Proc. Edinburgh Math. Soc., {\bf 59} (2016), 11--16.


\bibitem{BHM} M.R. Bridson, A. Hinkkanen and G. Martin,
{\em Quasiregular self-mappings of manifolds and word hyperbolic groups},
Compositio Math. {\bf{143}} (2007), 1613--1622.

\bibitem{BCT} B.A. Burton, A. Coward and S. Tillmann, {\em Computing closed essential  surfaces in knot complements},
 	SCG '13: Proceedings of the 29th Annual Symposium on Computational Geometry, ACM, 2013, pp. 405--414.

\bibitem{hatcher} Allen Hatcher, {\em Boundary curves of incompressible surfaces},
Pac. J. Math. {\bf 99} (1982), 373--37. 
 
 
 \bibitem{Hirsh} R. Hirshon, {\em  Some properties of Endomorphisms in residually finite groups},
 J. Austr. Math. Soc.  {\bf 24} (1977),  117--120.
 
 \bibitem{HW} S. Henry and J. Weeks, {\em Symmetry groups of hyperbolic knots and links},
  J. Knot Theory Ramifications, 1 (1992), 185--201
  
  \bibitem{KS} K. Kodama and M. Sakuma, {\em Symmetry groups of prime knots up to 10 crossings},
  in ``Knots 90" (Akio Kawauchi, ed.)
  Proceedings of the International Conference on Knot Theory and Related Topics held in Osaka (Japan). De Gruyter, Berlin, Boston 1992.  
  
  \bibitem{koji}
S. Kojima, {\em Isometry transformations of hyperbolic 3-manifolds}, Topol. Appl. {\bf{29}} (1988),
297--307.
 
\bibitem{K} Ilya Kapovich, {\em A non-quasiconvexity embedding theorem for hyperbolic groups},
Math.Proc.Camb.Phil.Soc., {\bf{127}} (1999), 461--486.

\bibitem{LR} Darren Long and Alan Reid, {\em On asymmetric hyperbolic manifolds},
Math.Proc.Camb.Phil.Soc {\bf 138} (2005), 301--306.

\bibitem{LS} Roger C. Lyndon and Paul E.  Schupp, {\it Combinatorial Group Theory},
Ergebnisse der Mathematik und ihrer Grenzgebiete 89, Springer--Verlag (Berlin), 1977.

\bibitem{menasco}
W. Menasco, 
{\em Closed incompressible surfaces in alternating knot and link complements}, Topology {\bf{23}} (1984), 37--44.

\bibitem{MS} Charles F. Miller  III and Paul E. Schupp, {\em Embeddings into Hopfian Groups},
J. Algebra {\bf 17} (1971), 171--176.

 
\bibitem {osin} D.~V. Osin, Small cancellations over relatively hyperbolic groups and embedding theorems, Ann. of Math., {\bf 172} (2010),  
1--39.
 
\bibitem{scott} G.P.~Scott, {\em Compact submanifolds of 3-manifolds}, J. London Mathematical Soc. {\bf{7}} (1973), 246--250.

\bibitem{serre} J-P Serre,
{\it{Trees}}, Springer-Verlag, Berlin--Heidelberg--New York, 1977.
  

\bibitem{stall} J.R.~Stallings,
{\em Group theory and three-dimensional manifolds}, Yale Mathematical Monographs {\bf{4}},
Yale University Press, New Haven, 1971. 

\bibitem{sela} Z.~Sela,
{\em Endomorphisms of hyperbolic groups : I}, Topology {\bf{38}} (1999), 301--321.

\bibitem{wpt} W.P. Thurston, {\em The Geometry and Topology of 3--manifolds}, Princeton University Lectures Notes, 1978--1981.

\end{thebibliography}
\end{document}